\documentclass[12pt]{nyjm}
\usepackage{amsfonts, amssymb, curves, epic, graphicx, amsmath,verbatim, amscd}
\begin{document}

\title{Quotient cohomology for tiling spaces}

\author{Marcy Barge} 
\address{Department of Mathematics, Montana State University,
Bozeman, MT 59717, USA}
\email{barge@math.montana.edu}

\author{Lorenzo Sadun}
\address{Department of Mathematics, University of Texas, 
Austin, TX 78712, USA}
\email{sadun@math.utexas.edu}
\thanks{The work of the second author is partially supported by 
NSF grant DMS-0701055}
\date{August 26, 2011}

%
%

\begin{abstract} 
We define a relative version of tiling cohomology for the purpose of
comparing the topology of tiling dynamical systems when one is a factor of the 
other.  We illustrate this with examples, and 
outline a method for computing the cohomology of 
tiling spaces of finite type. 
\end{abstract} 

\maketitle

\newcommand{\R}{\mathbb{R}}
\newcommand{\N}{\mathbb{N}}
\newcommand{\Q}{\mathbb{Q}}
\newcommand{\Z}{\mathbb{Z}}
\newcommand{\cT}{\mathbb{T}}
\newcommand{\T}{\mathcal{T}}
\newcommand{\W}{\mathcal{W}}
\newcommand{\A}{\mathcal{A}}
\newcommand{\B}{\mathcal{B}}
\newcommand{\C}{\mathcal{C}}
\newcommand{\cL}{\mathcal{L}}
\newcommand{\M}{\mathcal{M}}
\newcommand{\cN}{\mathcal{N}}
\newcommand{\cH}{\mathcal{H}}
\newcommand{\cP}{\mathcal{P}}
\newcommand{\cS}{\mathcal{S}}
\newcommand{\V}{\mathcal{V}}
\newcommand{\F}{\mathcal{F}}
\newcommand{\G}{\mathcal{G}}
\newcommand{\Cech}{{\v Cech}{} }
\newcommand{\Om}{\Omega}
\newcommand{\cU}{\mathcal{U}}
\newcommand{\cV}{\mathcal{V}}

\newcommand{\vp}{\varphi}
\newcommand{\tv}{\tilde{\varphi}}
\newcommand{\tA}{\tilde{\A}}
\newcommand{\tT}{\tilde{T}}
\newcommand{\tf}{\tilde{f}}
\newcommand{\tsigma}{\tilde{\sigma}}
\newcommand{\hf}{\hat{f}}
\newcommand{\td}{\tilde}
\newcommand{\eps}{\epsilon}
\newcommand{\oK}{\overline{K}}
\newcommand{\half}{\frac{1}{2}}
\newcommand{\quarter}{\frac{1}{4}}

\newcommand{\larr}{\left( \begin{array}{c}}
\newcommand{\rarr}{\end{array} \right) }

\newcommand{\lsqarr}{\left[ \begin{array}{c}} 
\newcommand{\rsqarr}{\end{array} \right]} 

\newcommand{\uv}{(u, v)^T}

\newcommand{\arrow}{\rightarrow}
\newcommand{\larrow}{\leftarrow}
\newcommand{\inv}{\varprojlim}
\newcommand{\dir}{\varinjlim}
\newcommand{\seabox}{{\framebox{$\searrow$}}}
\newcommand{\neabox}{{\framebox{$\nearrow$}}}
\newcommand{\swabox}{{\framebox{$\swarrow$}}}
\newcommand{\nwabox}{{\framebox{$\nwarrow$}}}
\newcommand{\neswarrow}{{\nearrow\mkern-18mu\swarrow}}
\newcommand{\nwsearrow}{{\nwarrow\mkern-18mu\searrow}}
\newcommand{\neswabox}{{\framebox{$\nearrow\mkern-18mu\swarrow$}}}
\newcommand{\nwseabox}{{\framebox{$\nwarrow\mkern-18mu\searrow$}}}
\newcommand{\ER}{{\hbox{\tiny ER}}}
\newcommand{\nebox}[4]{\framebox[24pt][c]{$
\begin{smallmatrix}
{\scriptscriptstyle #3}\\
{\scriptscriptstyle #1}\!{\displaystyle\nearrow} {\scriptscriptstyle #2}\\
{\scriptscriptstyle #4}\end{smallmatrix}
$}}
\newcommand{\sebox}[4]{\framebox[24pt][c]{$
\begin{smallmatrix}
{\scriptscriptstyle #3}\\
{\scriptscriptstyle #1}\!{\displaystyle\searrow} {\scriptscriptstyle #2}\\
{\scriptscriptstyle #4}\end{smallmatrix}
$}}
\newcommand{\nwbox}[4]{\framebox[24pt][c]{$
\begin{smallmatrix}
{\scriptscriptstyle #3}\\
{\scriptscriptstyle #1}\!{\displaystyle\nwarrow} {\scriptscriptstyle #2}\\
{\scriptscriptstyle #4}\end{smallmatrix}
$}}
\newcommand{\swbox}[4]{\framebox[24pt][c]{$
\begin{smallmatrix}
{\scriptscriptstyle #3}\\
{\scriptscriptstyle #1}\!{\displaystyle\swarrow} {\scriptscriptstyle #2}\\
{\scriptscriptstyle #4}\end{smallmatrix}
$}}
\newcommand{\neswbox}[4]{\framebox[24pt][c]{$
\begin{smallmatrix}
{\scriptscriptstyle #3}\\
{\scriptscriptstyle #1}\!{\displaystyle\neswarrow} {\scriptscriptstyle #2}\\
{\scriptscriptstyle #4}\end{smallmatrix}
$}}
\newcommand{\nwsebox}[4]{\framebox[24pt][c]{$
\begin{smallmatrix}
{\scriptscriptstyle #3}\\
{\scriptscriptstyle #1}\!{\displaystyle\nwsearrow} {\scriptscriptstyle #2}\\
{\scriptscriptstyle #4}\end{smallmatrix}
$}}
\newcommand{\neboxx}[2]{\nebox{\phantom{1}}{\phantom{1}}{#1}{#2}}
\newcommand{\nwboxx}[2]{\nwbox{\phantom{1}}{\phantom{1}}{#1}{#2}}
\newcommand{\seboxx}[2]{\sebox{\phantom{1}}{\phantom{1}}{#1}{#2}}
\newcommand{\swboxx}[2]{\swbox{\phantom{1}}{\phantom{1}}{#1}{#2}}
\newcommand{\neswboxx}[2]{\neswbox{#1}{#2}{\phantom{1}}{\phantom{1}}}
\newcommand{\nwseboxx}[2]{\nwsebox{#1}{#2}{\phantom{1}}{\phantom{1}}}
\newcommand{\neboxxx}{\neboxx{\phantom{1}}{\phantom{1}}}
\newcommand{\nwboxxx}{\nwboxx{\phantom{1}}{\phantom{1}}}
\newcommand{\seboxxx}{\seboxx{\phantom{1}}{\phantom{1}}}
\newcommand{\swboxxx}{\swboxx{\phantom{1}}{\phantom{1}}}
\newcommand{\neswboxxx}{\neswboxx{\phantom{1}}{\phantom{1}}}
\newcommand{\nwseboxxx}{\nwseboxx{\phantom{1}}{\phantom{1}}}

\newcommand{\vvec}[2]{\left ( \begin{smallmatrix} #1 \cr #2 \end{smallmatrix}
\right )}
\newcommand{\vvvec}[3]{\left ( \begin{smallmatrix} #1 \cr #2 \cr #3 
\end{smallmatrix}
\right )}
\newcommand{\vvvvec}[4]{\left ( \begin{smallmatrix} #1 \cr #2 \cr #3 \cr #4
\end{smallmatrix}
\right )}

\newtheorem{theorem}{Theorem}
\newtheorem{cor}[theorem]{Corollary}
\newtheorem{corollary}[theorem]{Corollary}
\newtheorem{lemma}[theorem]{Lemma} 
\newtheorem{prop}[theorem]{Proposition}
\theoremstyle{definition}
\newtheorem{example}{Example}
\newtheorem{remark}{Remark}

\section{Introduction}
Since its development, cohomology has been an
essential tool of algebraic topology. It is a topological invariant that can
tell spaces apart (both with the groups and with the ring structure). It is
computable by a variety of cut-and-paste rules. 
It is a functor that relates two or more 
spaces and the maps between them. Finally, it is the setting for other 
topological structures, such as characteristic classes. 

The cohomology of tiling spaces is far less developed, and in some ways
resembles the state of abstract cohomology in the mid-20th century. Mostly it
has been used to tell spaces apart. 
There has been little progress in using
cut-and-paste arguments to compute anything, and most computations have 
relied on inverse limit structures. It is only used to study one space at a
time, not in a functorial setting. We have a limited understanding of 
what cohomology tells us, and what other problems can be addressed using
cohomology. (However, see \cite{B, BBG, CGU, CS, ER} 
for some applications to gap labeling, 
deformations, spaces of measures, and exact regularity.) 

This paper is an attempt to remedy this deficiency. By specializing the 
algebraic mapping cylinder and mapping cone construction to tiling theory, 
we develop a relative version of tiling cohomology, which we call {\em
quotient cohomology}. We then show how to use quotient cohomology 
to relate similar tiling spaces. 

In Section \ref{defs}, we lay out the definitions and basic properties 
of quotient cohomology. In Section \ref{examples} we illustrate the
formalism with some simple examples, both from basic topology and from
one dimensional tilings. In 
Section \ref{tools} we develop the tools needed to handle more
complicated problems. The key tool for tiling theory is Proposition 
\ref{Cohomology.of.suspension.applied}, which describes how to get
the quotient cohomology of two tiling spaces that differ only on the
suspension of a lower-dimensional tiling space. In Section \ref{chairs}
we examine a family of nine tiling spaces that includes the 2-dimensional
dyadic solenoid and the ``chair'' substitution tiling. By applying
Proposition 4 repeatedly, we relate the cohomology of each space to that
of the dyadic solenoid.  Finally, in Section \ref{finite} we explore
the cohomology of tiling spaces of finite type, a class of tiling spaces
that has previously defied analysis.

\section{Definitions}\label{defs}

A {\em tiling} of $\R^d$ is a collection of closed topological
disks, called tiles, such that tiles overlap only on their boundaries
and such that the union of all the tiles is $\R^d$. In addition to
their position and geometric shape, tiles may carry labels. The translation
group $\R^d$ transforms a tiling into a different tiling 
by moving all tiles simultaneously. If $\T$ is a tiling, then $\T -v$ is
the tiling translated by $v \in \R^d$. We endow the orbit of $\T$ under
translation with a metric where two tilings are $\epsilon$-close if they agree,
up to a translation by $\epsilon$ or less, on a ball of radius $\epsilon^{-1}$
around the orgin. The completion $X_\T$ of the orbit of $\T$ is called 
the {\em hull} of $\T$, or the {\em tiling space} associated with
$\T$. 
Locally, $X_\T$ is the product of $\R^d$ with a totally disconnected 
space, typically a Cantor set. 
$X_\T$, equipped with the action of the translation group
$\R^d$, is a {\em tiling dynamical system}. 
Most of the tiling dynamical systems in the literature 
are compact, minimal and uniquely ergodic.

Substitutions provide an important method for the generation of
tilings. A {\em $d$-dimensional substitution} is a recipe that
linearly inflates each of a finite collection of $d$-dimensional
prototiles and specifies a tiling of each of the inflated prototiles
by translates of the prototiles. A tiling of $\R^d$ obtained as a
limit of repeated application of a substitution is called a {\em
  substitution tiling} and its hull is a {\em substitution tiling
  space}. Under mild assumptions (\cite{sol}), a substitution induces
a {\em substitution homeomorphism} on its tiling space.

Besides their intrinsic interest, tiling dynamical systems model a variety of
structures in dynamics; for instance, every 1-dimensional orientable
expanding attractor is topologically conjugate to either the shift
homeomorphism on a solenoid or the substitution homeomorphism on a
substitution tiling space \cite{ap}.  For background information on
tiling spaces and their topology, see \cite{book}.

If $X$ and $Y$ are topological spaces and $f: X \to Y$ is an
injection, then the relative (co)homology groups $H_k(Y,X)$ and
$H^k(Y,X)$ relate the (co)homology of $X$ and $Y$ via long exact sequences
\begin{eqnarray*}
&& \cdots \to H_{k+1}(Y,X) \to H_k(X) \xrightarrow{f_*} H_k(Y) \to H_k(Y,X) 
\to \cdots, \\
&& \cdots \to H^k(Y,X) \to H^k(Y) \xrightarrow{f^*} H^k(X) 
\to H^{k+1}(Y,X) \to \cdots. 
\end{eqnarray*}
Factor maps between minimal dynamical systems are surjections, since
the image of each orbit is dense, but typically are not
injections. To study such spaces, we need a different tool.

Let $f:X \to Y$ be a quotient map such that the pullback $f^*$ is
injective on cochains. This is the typical situation for covering
spaces, for branched covers, and for factor maps
between tiling spaces. When dealing with tiling spaces, ``cochains'' 
can either mean \Cech
cochains or pattern-equivariant cochains \cite{K, KP, integer}; 
our arguments apply equally well to both. Define the
cochain group $C_Q^k(X,Y)$ to be $C^k(X)/f^*(C^k(Y))$.  
The usual coboundary operator
sends $C_Q^k(X,Y)$ to $C_Q^{k+1}(X,Y)$, and we define the {\em quotient
cohomology} $H_Q^k(X,Y)$ to be the kernel of the coboundary modulo the image.
By the snake lemma,
the short exact sequence of
cochain complexes
\begin{equation*} 0 \to C^k(Y) \xrightarrow{f^*} C^k(X) \to C_Q^k(X,Y) \to 0
\end{equation*}
induces a long exact sequence 
\begin{equation}\label{LES1} \cdots \to H^{k-1}_Q(X,Y) \to 
H^k(Y) \xrightarrow{f^*} H^k(X) 
\to H^k_Q(X,Y) 
\to \cdots
\end{equation}
relating the cohomologies of $X$ and $Y$ to $H_Q^*(X,Y)$. 

Quotient cohomology is related to an ordinary relative cohomology
group involving the mapping cylinder $M_f = (X \times [0,1])\coprod Y/\sim$,
where $(x,1)\sim f(x)$, or to the reduced cohomology of a mapping
cone, where we collapse $X \times \{0\} \subset M_f$ to a single point. 
$M_f$ is homotopy equivalent to $Y$, and the inclusion
$i:X \to M_f$, $i(x)=(x,0)$ is homotopically the same as $f$. This yields
the (standard) long exact sequence in relative cohomology
\begin{equation}\label{LES2}
\cdots H^k(M_f,X) \to H^k(M_f) \xrightarrow{i^*} H^k(X) \to H^{k+1}(M_f,X)
\to \cdots.
\end{equation}
Applying the Five Lemma to 
the long exact sequences (\ref{LES1}) and (\ref{LES2}) 
and noting that
$H^k(M_f) \simeq H^k(Y)$, with $i^*$ essentially the same as $f^*$,
we see that $H_Q^k(X,Y)$ equals $H^{k+1}(M_f, X)$.  

Quotient cohomology can also be viewed as the cohomology of the
algebraic mapping cone of $X$ and $Y$ \cite{Weibel}.
Specifically, let $C^k_f
= C^k(X) \oplus C^{k+1}(Y)$, and 
let $d_f(a,b)=(d_X(a)+f^*(b), -d_Y(b))$. The cohomology of $d_f$ fits into
the same exact sequence as $H^k_Q(X,Y)$, and hence is isomorphic to
$H^k_Q(X,Y)$. Indeed, the mapping cone construction works even when $f^*$ is
not injective at the level of cochains. 

The mapping cylinder and cone constructions are extremely general. 
They are also cumbersome, and to the best of our knowledge 
have never been used in tiling theory. 
Indeed, many of 
the structures defined for tiling spaces, such as pattern-equivariant
cohomology \cite{K,KP}, rely on an identification of certain features of
a tiling $\T$ with sets of tilings in $X_\T$. These structures 
make no sense on a (topological) mapping cylinder. 
Fortunately, quotient cohomology does make sense, and 
provides an easy yet powerful tool for studying the topology of tiling
spaces.

\section{Topological and tiling examples}\label{examples}

\subsection{Basic topological examples}

\begin{example}
Let $Y$ be a CW complex with a distinguished $n$-cell $e^n$ that is
not on the boundary of any cell of higher dimension. Let $X$ be the same
complex, only with two copies of $e^n$ (call them $e^n_1$ and $e^n_2$), each
with the same boundary as $e^n$, and let $f$ be the map that identifies
$e^n_1$ and $e^n_2$. Then, working with cellular cohomology, 
$C_Q^k(X,Y)$ is trivial in all dimensions except
$k=n$, and $C_Q^n(X,Y)$ is generated by the duals $(e^n_i)'$ to $e^n_i$, 
with the relation $(e^n_1)' + (e^n_2)' = 0$, so $H_Q^k(X,Y) = \Z$ if $k=n$
and is zero otherwise. 

Slightly more generally, let $X$ be a CW complex and let $Y$ be the quotient
of $Y$ by the identification of two $n$-cells $e^n_{1,2}$ of $X$, whose 
boundaries have previously been identified. (The generalization is that we
make no assumptions about how higher-dimensional cells attach to $e^n_{1,2}$.)
Then, as before, $C_Q^k(X,Y) = \Z$ when $k=n$ and vanishes otherwise, so
$H_Q^k(X,Y) = \Z$ for $k=n$ and vanishes otherwise. 
Up to homotopy, identifying $e^n_{1,2}$ is the same thing as gluing in an
$(n+1)$-cell with boundary $e^n_1 - e^n_2$, in which case $f$ can be viewed
as an inclusion into a space $Y'$ that is homotopy equivalent to $Y$, and 
$H^k_Q(X,Y) = H^{k+1}(Y',X)$. 
\end{example}

Repeating the construction as needed, we can compute the quotient
cohomology of any two CW complexes $X$ and $Y$, where $Y$ is the quotient
of $X$ by identification of some cells. 

\begin{figure}[h]
\includegraphics[width=2.5in]{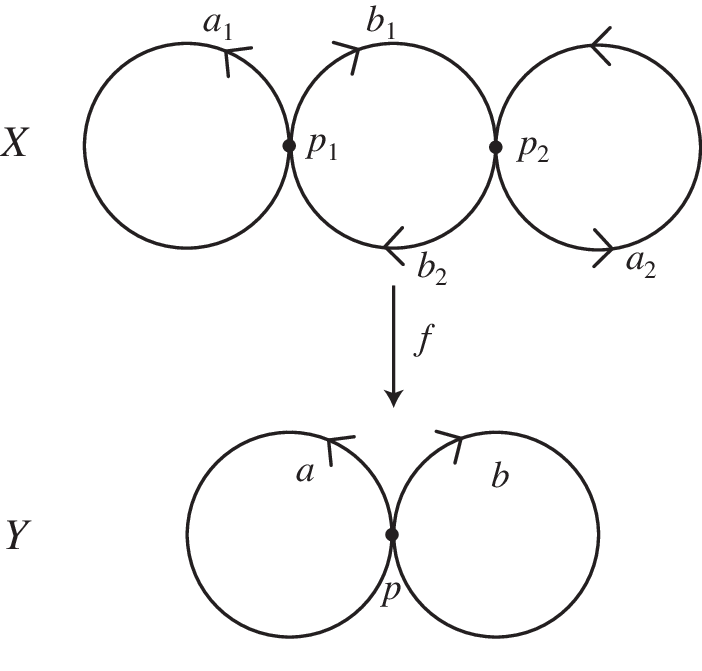}
\caption{A simple example of quotient cohomology.}\label{fig1}
\end{figure}

\begin{example}\label{figure8}
Figure \ref{fig1} shows two graphs, with $X$ the double cover of $Y$. 
Let $f$ be the covering map, 
sending each edge $a_i$ to $a$, each $b_i$ to $b$, and each vertex $p_i$ to $p$.
Since 
$f^*(p')=p_1'+p_2'$, $f^*(a')=a_1'+a_2'$ and $f^*(b')=b_1'+b_2'$, 
$C_Q^0(X,Y)=\Z$ is
generated by $p_1'$, with $p_2'=-p_1'$, while $C_Q^1(X,Y)=\Z^2$ is generated by 
$a_1'$ and $b_1'$, with $a_2'=-a_1'$ and $b_2'=-b_1'$. The coboundary of 
$p_1'$ is $b_2'-b_1'=-2b_1'$, so $H_Q^0(X,Y)=0$ and $H_Q^1(X,Y)=\Z \oplus \Z_2$, 
with generators $a_1'$ and $b_1'$. Our long exact sequence (\ref{LES1}) is then
\begin{equation}\label{fig8}
0 \to \Z \xrightarrow{f^*} \Z \to 0 \to \Z^2 \xrightarrow{f^*} \Z^3 \to
\Z \oplus \Z_2 \to 0.
\end{equation}
Torsion appears in $H^1_Q(X,Y)$, reflecting the fact that $f^*(b')$ is 
cohomologous to $2b_1' \in H^1(X)$. 
\end{example}


\subsection{One-dimensional tiling examples}\label{1D}

\begin{figure}[h]
\includegraphics[width=3in]{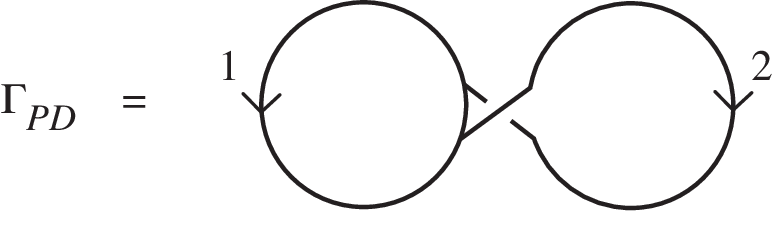}
\caption{The approximant for the period-doubling substitution tiling}
\label{fig2}
\end{figure}

\begin{example}[Period Doubling over the 2-Solenoid]
The period doubling substitution is $1 \to 21$, $2 \to 11$. (By this we mean that there are two prototiles, each of the same length, one labeled 1 and the other 2. The substitution inflates each prototile by a factor of two, tiling the inflated 1 with a 2 and a 1, and the inflated 2 by two 1's.) Since this is
a substitution of constant length 2, there is a natural map from
the period doubing
tiling space $\Omega_{PD}$ to the dyadic solenoid $S_2$. 
$\Omega_{PD}$ can be written as the inverse limit via substitution of the
approximant $\Gamma_{PD}$ shown in Figure \ref{fig2}, where the long 
edges 1 and 2 represent tile types and the short edges represent possible
transitions \cite{BD}. 
$\Gamma_{PD}$ is homotopically a figure 8, and maps to a circle by identifying
the two long edges and identifying the three short edges.
This projection of $\Gamma_{PD}$ to the circle intertwines the 
substitution on $\Gamma_{PD}$ and the doubling map on $S^1$, and 
has quotient cohomology
$H^1_Q(\Gamma_{PD}, S^1)=\Z$ (and $H^0_Q=0$).
The dyadic solenoid $S_2$ is the inverse limit of a circle under doubling,
and $H^k_Q(\Omega_{PD},S_2)$ is the direct limit of $H^k_Q(\Gamma_{PD},S^1)$
under substitution.
Substitution acts 
on $H^1_Q(\Gamma_{PD}, S^1)$ by multiplication
by $-1$,  so $H^1_Q(\Omega_{PD}, S_2) = \dir H^1_Q(\Gamma_{PD},S^1)=\Z$. 
\end{example}

\begin{example}
The Thue-Morse substitution tiling of the real line is well known to be the
double cover of the period-doubling tiling. Here we explore the
quotient cohomology of the pair. 

The Thue-Morse substitution is $A \to AB$, $B \to BA$. We can rewrite this
in terms of collared tiles, distinguishing between $A$ tiles that are followed
by $B$ tiles (call these $A_1$) and $A$ tiles that are followed by $A$ tiles
(call these $A_2$). Likewise, $B$ tiles that are followed by $A$ tiles are
called $B_1$ and $B$ tiles that are followed by $B$ tiles are $B_2$. In terms
of these collared tiles, the substitution is:
\begin{equation}\label{TMcollarsub} A_1 \to A_1B_2; \qquad A_2 \to A_1 B_1; \qquad
B_1 \to B_1A_2; \qquad B_2 \to B_1A_1. 
\end{equation}
The map from the Thue-Morse substitution space to the period-doubling space
just replaces each $A_1$ or $B_1$ tile with a 1 and each $A_2$ or $B_2$ with
a 2. This is exactly 2:1, and the preimage of any period-doubling tiling 
consists of a Thue-Morse tiling, plus a second tiling obtained by swapping 
$A_i \leftrightarrow B_i$ at each place.  

\begin{figure}[h]
\includegraphics[width=2.5in]{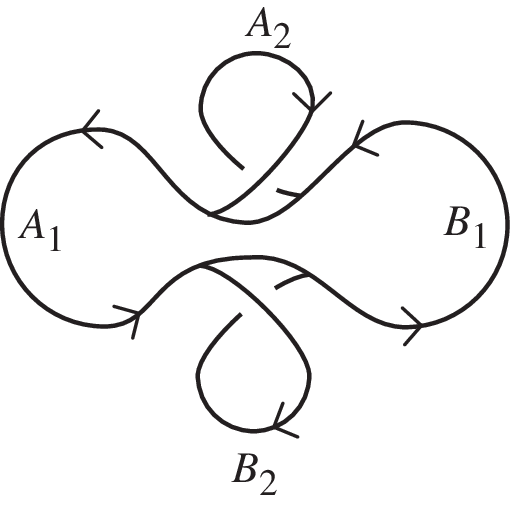}
\caption{The approximant $\Gamma_{TM}$ for the Thue-Morse substitution tiling space.}\label{fig3}
\end{figure}

\begin{figure}[h]
\includegraphics[width=2.5in]{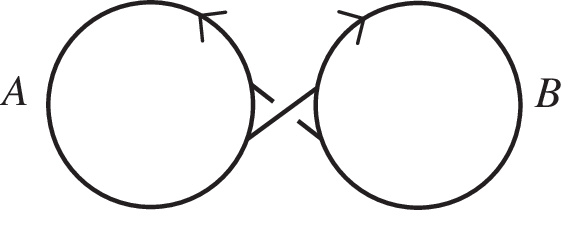}
\caption{Another approximant $\Gamma_{TM'}$ 
for the Thue-Morse tiling space.}\label{fig4} 
\end{figure}

Using collared tiles, we obtain the Thue-Morse tiling space 
$\Omega_{TM}$ as the inverse limit, under the 
substitution (\ref{TMcollarsub}), 
of the approximant $\Gamma_{TM}$ shown in 
Figure \ref{fig3}.  
\footnote{The space $\Omega_{TM}$ is more frequently computed as 
the inverse limit
of the simpler approximant $\Gamma_{TM'}$, 
shown in Figure \ref{fig4}.}
$\Gamma_{TM}$ is homotopy equivalent to the 
double cover of a figure 8, just as $\Gamma_{PD}$
is equivalent to a figure 8. Indeed, 
the quotient map from $\Gamma_{TM}$ to $\Gamma_{PD}$ is, 
up to homotopy, the 
covering map of the figure 8 that we studied in Example \ref{figure8}, with
$H^1_Q(\Gamma_{TM}, \Gamma_{PD}) = \Z \oplus \Z_2$, 
with $A_1'$ (or $B_1'$) generating
the $\Z_2$ factor and $A_2'$ (or $B_2'$) generating the $\Z$ factor.

Under substitution, $A_1'+A_2'$ pulls back to $A_1'+B_1'+A_2'+B_2'=0$,
while $A_1'$ pulls back to $A_1'+A_2'+B_2'=A_1'$, and 
$H_Q^1(\Omega_{TM},\Omega_{PD})= \dir H_Q^1(\Gamma_{TM}, \Gamma_{PD})=\Z_2$.  
The groups $H^1(\Omega_{TM})$
and $H^1(\Omega_{PD})$ are both isomorphic to $\Z[{\textstyle \half}] \oplus \Z$ and
the exact sequence (\ref{LES1}) applied to $\Omega_{TM}$ and
$\Omega_{PD}$ is
\begin{equation} 
0 \to \Z \xrightarrow{f^*} \Z \to 0 \to \Z[{\textstyle \half}] \oplus \Z \xrightarrow{f^*}
\Z[{\textstyle \half}] \oplus \Z \to \Z_2 \to 0.
\end{equation}
Although $H^1(\Omega_{TM})$ and $H^1(\Omega_{PD})$ are 
isomorphic as abstract groups, the pullback map $f^*$ is not an 
isomorphism. Rather, it is the identity on $\Z[{\textstyle \half}]$ and 
multiplication by 2 on $\Z$. 
\end{example}

The remaining one dimensional examples may seem trivial or 
contrived, but they are the building
blocks for understanding the 2-dimensional examples that follow.

\begin{example}[Degenerations A and B.]
If $X = S_2 \times \{1,2\}$ is 2 copies of a dyadic solenoid and $Y=S_2$ 
is a single copy, and if $f$ is projection onto the first factor,
then $H^1_Q(X,Y)= H^1(S_2)=\Z[{\textstyle \half}]$ and $H^0_Q(X,Y)=H^0(S_2)=\Z$. We call this
degeneration A. 
Degeneration B is where $X = \Omega_{PD} \times \{1,2\}$ projects to 
$Y=\Omega_{PD}$, in which case  
$H^1_Q(X,Y)= H^1(\Omega_{PD}) = \Z[{\textstyle \half}]\oplus \Z$ and 
$H^1_Q(X,Y)=H^0(\Omega_{PD})=\Z$. 
\end{example}

\begin{example}[Degeneration C.]  The space $\Gamma_{TM'}$ of Figure
\ref{fig4} also serves as an approximant for another tiling space of 
interest using a different substitution map. Let $X$ be the inverse limit of the
$\Gamma_{TM'}$ under a map that wraps each large circle twice around itself, and
that doubles the length of the small intervals that link the circles.
That is, the interval that goes from the left circle to the right one
turns into a piece of the left circle followed by the interval, followed
by a piece of the right circle.  Note that the small loop obtained from
the four small intervals is homologically invariant under this map. 

Let $Y=S_2$ be the dyadic solenoid, viewed as the inverse limit of a
circle under doubling. The obvious map from $\Gamma_{TM'}$ to $S^1$ has
$H^1_Q(\Gamma_{TM'} ,S^1)=\Z \oplus \Z$ and $H^0_Q(\Gamma_{TM'}, S^1)=0$. 
Substitution multiplies the
first factor in $H^1_Q$ by 2 and the second factor by 1, so $H^1_Q(X,Y)=
\dir H^1_Q(\Gamma_{TM'}, S^1) = \Z[{\textstyle \half}]
\oplus \Z$, while $H^0_Q(X,Y)=0$.
\end{example}

\section{Tools}\label{tools}

Suppose that $f:X\to\ Y$ and $g:Y\to Z$ are quotient maps that induce
injections on cochains. Then $h:=g\circ f:X\to Z$ is also such a map
and there is then a short exact sequence of the corresponding chain
complexes
\begin{equation} 0 \to C_Q^*(Y,Z) \to C_Q^*(X,Z) \to C_Q^*(X,Y) \to 0
\end{equation}
which induces the {\em long exact sequence for the triple}
\begin{equation}\label{LEST} \cdots \to H^k_Q(Y,Z) \to H^k_Q(X,Z) 
\to H^k_Q(X,Y)\to H^{k+1}_Q(Y,Z)
\to \cdots.
\end{equation}

\begin{theorem}(Excision) Suppose that $f: X \to Y$ is a quotient map
  that induces an injection on cochains. Suppose that $Z \subset X$ is
  an open set such that $f|_{\bar Z}$ is a homeomorphism onto its
  image. Then the inclusion induced homomorphism from $H^*_Q(X,Y)$ to
  $H^*_Q(X\setminus Z, Y\setminus f(Z))$ is an isomorphism.
\end{theorem}

\begin{proof}
  Inclusions of $X$ and $Y$ into $M_f$ (as $X\times\{0\}$ and
  $Y\times\{1\}$) induce a homomorphism from the long exact sequence
  for the pair $(X,Y)$ in the quotient cohomology to the usual long
  exact sequence for the pair $(M_f,X\times\{0\})$. The induced
  homomorphism from $H_Q^k(X,Y)$ to $H^{k+1}(M_f,X\times\{0\})$ is an
  isomorphism, by the five lemma.  Since $f|_{\bar Z}$ is a
  homeomorphism onto its image, inclusion of $X\times\{0\}$ into
  $X\times\{0\}\cup(\bar{Z}\times [0,1])\subset M_f$ is a homotopy
  equivalence. This inclusion then induces an isomorphism from
  $H^{k+1}(M_f,X\times\{0\})$ onto
  $H^{k+1}(M_f,X\times\{0\}\cup(\bar{Z}\times [0,1]))$. Since
  $f|_{\partial Z}$ is a homeomorphism onto its image, inclusion of
  $(X\times\{0\})\setminus (Z\times\{0\})$ into
  $((X\times\{0\})\setminus (Z\times\{0\}))\cup (\partial Z\times
  [0,1])\subset M_f$ is a homotopy equivalence which then induces an
  isomorphism from $H^{k+1}(M_f\setminus(Z\times [
  0,1]),(X\times\{0\})\setminus (Z\times\{0\}))$ onto
  $H^{k+1}(M_f\setminus (Z\times [0,1]),((X\times\{0\})\setminus
  (Z\times\{0\}))\cup (\partial Z\times [0,1]))$. By ordinary
  excision, the inclusion of $(M_f\setminus(Z\times [0,1]),((X\times
  \{0\}\ ) \setminus(Z\times\{0\}))\cup (\partial Z\times [0,1]))$
  into $(M_f,(X\times\{0\} )\cup(\bar{Z}\times [0,1]))$ induces an
  isomorphism from $H^{k+1}(M_f,X\times\{0\}\cup(\bar{Z}\times
  [0,1]))$ onto $H^{k+1}(M_f\setminus (Z\times
  [0,1]),((X\times\{0\})\setminus (Z\times\{0\}))\cup (\partial
  Z\times [0,1]))\simeq H^{k+1}(M_f\setminus (Z\times [0,1]),(X\times
  \{0\}\ ) \setminus(Z\times\{0\}))$. The latter group is just
  $H^{k+1}(M_{f|_{X\setminus Z}},(X\setminus{Z})\times\{0\})$, which
  is (inclusion induced) isomorphic with $H^k_Q(X\setminus Z,
  Y\setminus f(Z))$.
 
\end{proof}
 
\begin{theorem}(Mayer-Vietoris Sequence)\label{Mayer-Vietoris
    Sequence} Suppose that $X_1$ and $X_2$ are subspaces of $X$ with
  $X$ the union of the interiors of $X_1$ and $X_2$. Suppose further
  that $f:X\to Y$, $f|_{X_1}$, $f|_{X_2}$, and $f|_{X_1\cap X_2}$ are
  all quotient maps onto $Y$ that induce injections on cochains. There
  is then a long exact sequence

\begin{eqnarray}
\label{MVS} \cdots \to H^k_Q(X,Y)& \to & H^k_Q(X_1,Y)\oplus 
  H^k_Q(X_2,Y) \cr
  &\to& H^k_Q(X_1\cap X_2,Y) \to H^{k+1}(X,Y)
  \to \cdots
\end{eqnarray}
\end{theorem}

\begin{proof} This is just the relative Mayer-Vietoris sequence for
  the pairs $(M_{f_1},X_1)$ and $(M_{f_2},X_2)$, with $f_i:=f|_{X_i}$,
  together with the identifications $H^k_Q(X_i,Y)\simeq
  H^{k+1}(M_{f_i},X_i)$, etc.
\end{proof}
Given $f:X\to Y$, let $S^k_f(X):= X\times\mathbb{D}^k/\sim$, where
$\mathbb{D}^k$ is the closed $k$-disk and $(x,v)\sim(y,v)$ for
$v\in\partial \mathbb{D}^k$ if $f(x)=f(y)$. The {\em $k$-fold
  fiber-wise suspension of $f$} is the map $S^k(f): S^k_f(X)\to Y$ by
$S^k(f)([(x,v)]):=f(x)$.

\begin{theorem}(Cohomology of Suspension)\label{Cohomology.of.Suspension} 
Suppose that $f:X\to Y$ is a quotient map that induces
  an injection on cochains. Then $H^{n+k}_Q(S^k_f(X),Y)\simeq
  H^n_Q(X,Y)$ for all $n$ and all $k\ge 0$.
\end{theorem}

\begin{proof}
  As $S^{j+1}_f(X)$ is homeomorphic with $S^1_{S^j(f)}(X)$, it
  suffices to prove the theorem with $k=1$. Let
  $X_{-1}:=X\times[-1,1/2]/\sim$ and $X_1:=X\times[1/2,1]/\sim$. Then
  $f|_{X_i}$ is a homotopy equivalence, so $H^*_Q(X_i,Y)={0}$ for
  $i=\pm 1$. Clearly, $H^*_Q(X_1\cap X_{-1},Y) \simeq H^*_Q(X,Y)$. The
  Mayer-Vietoris sequence gives the result.
\end{proof}

If $X$ is an $n$-dimensional tiling space, $X_1$ is a closed subset of
$X$, and $\Gamma$ is a $k$-dimensional subspace of $\R^n$, we will say
that $X_1$ is a {\em $k$-dimensional tiling subspace of} $X$ {\em in
  the direction of } $\Gamma$ provided if $T\in X_1$ then $T-v\in X_1$
if and only if $v\in \Gamma$. If $X_1$ is a $k$-dimensional tiling
subspace of $X$ in the direction of $\Gamma$ and $\sim$ is an
equivalence relation on $X_1$, we will say that $\sim$ is {\em
  uniformly asymptotic} provided for each $\epsilon >0$ there is an
$R$ so that if $T,T'\in X_1$ and $T\sim T'$, then
$d(T-v,T'-v)<\epsilon$ for all $v\in \Gamma^{\perp}$ with $|v|\ge R$.

\begin{prop}\label{Cohomology.of.suspension.applied}
  Suppose that $X$ is a non-periodic $n$-dimensional tiling space and
  $f:X\to Y$ is an $\mathbb{R}^n$-equivariant quotient map that
  induces an injection on cochains. Suppose also that $X'$ is a
  $k$-dimensional tiling subspace of $X$ in the direction of $\Gamma$
  and let $Y'=f(X')$. Let $\sim$ be the relation on $X'$ defined by
  $T\sim T'$ if and only if $f(T)=f(T')$: assume that $\sim$ is
  uniformly asymptotic. In addition, assume that $f$ is one-to-one off
  $X'-\mathbb{R}^n:=\{T-v:T\in X',v\in\mathbb{R}^n\}$ and that if
  $T,T'\in X'$ and $v\in\mathbb{R}^n$ are such that $f(T'-v)=f(T)$,
  then $v\in\Gamma$. Then $H^m_Q(X,Y)\simeq H^{m-n+k}_Q(X',Y')$.
\end{prop}

\begin{proof}
  For $r\ge 0$, let $\sim_r$ be defined on $X$ by $T_1\sim_r T_2$ if
  and only if $f(T_1)=f(T_2)$ and either $T_1=T_2$ or there is
  $v\in\Gamma^{\perp}$, with $|v|\ge r$, so that $T_1-v$ and $T_2-v$
  are in $X'$. Then $\sim_r$ is a closed equivalence relation. Let
  $X_r:=X/\sim_r$ and, for $r_1\le r_2$, let $p_{r_2,r_1}:X_{r_2}\to
  X_{r_1}$ be the natural quotient map. Then $X\simeq \inv
  p_{r_2,r_1}$ and $H^*_Q(X,Y)\simeq \dir p^*_{r_2,r_1}$. Moreover,
  $p_{r_2,r_1}$ is a homotopy equivalence for $r_2\ge r_1>0$ so
  $H^*_Q(X,Y)\simeq H^*_Q(X_1,Y)$, where $f_1:X_1\to Y$ is given by
  $f_1([T]):=f(T)$. Let $Z:= f_1^{-1}(Y\setminus
  f(X'-\mathbb{D}^{n-k}))$. Then $f_1$ is one-to-one on $\bar{Z}$ and
  $H^*_Q(X_1,Y)\simeq H^*_Q(X_1\setminus Z,Y\setminus f_1(Z))$ by
  excision. Now $X_1\setminus Z\simeq S^{n-k}_{f|_{X'}} (X')$ and $Y'$
  is a deformation retract of $Y\setminus f_1(Z)$ (the latter follows
  from the hypothesis that if $T,T'\in X'$ and $v\in\mathbb{R}^n$ are
  such that $f(T'-v)=f(T)$, then $v\in\Gamma$). Thus
  $H^*_Q(X_1\setminus Z, Y\setminus f_1(Z))\simeq
  H^*_Q(S^{n-k}_{f|_{X'}}(X'),Y')$ and the proposition follows from
  Theorem \ref{Cohomology.of.Suspension}.
\end{proof}






\begin{example}
The map of the period-doubling substitution space $\Omega_{PD}$
to the 2-solenoid $S_2$ fits into the framework of Proposition
\ref{Cohomology.of.suspension.applied}. The map is 1:1 except
on two doubly asymptotic $\R$-orbits 
that are identified. That is, $n=1$, $k=0$, 
$X'$ is a two-point set, and $Y'$ is a single point, so
$H^1_Q(\Omega_{PD},S_2)=\Z$, as computed earlier. 

Likewise, the map from the (two-dimensional) half-hex tiling space
$\Omega_{hh}$ to the two-dimensional dyadic solenoid $S_2 \times S_2$
is 1:1 except on three $\R^2$-orbits.
In this case $n=2$, $k=0$, $X'$ is a three-point set, 
and $Y'$ is a single point,
so $H^2_Q(\Omega_{hh},S_2\times S_2)=\Z^2$, while $H^1_Q=H^0_Q=0$.
\end{example}


\section{Variations on the chair tiling}\label{chairs}

It frequently happens that one tiling space is a factor
of another, and that the factor map is almost-everywhere 1:1. 
For instance, the chair tiling space that has the 2-dimensional dyadic solenoid
as an almost-1:1 factor. 
In addition, Mozes \cite{Mozes} and 
Goodman-Strauss \cite{Chaim.annals}
have proven that every substitution tiling space in dimension 2 and higher,
meeting
some mild conditions, is an almost-1:1 factor of a tiling space obtained
from local matching rules. 

These examples do not fit directly into the
framework of Proposition \ref{Cohomology.of.suspension.applied}. However,
it is possible to expand the chair example to make it fit. The chair and
the dyadic solenoid belong to a family of nine tiling spaces, connected by
simpler factor maps such that 
Proposition \ref{Cohomology.of.suspension.applied}
applies to each such map.

\subsection{The nine models}

Each model comes from a substitution. The simplest of these
is the 2-dimensional 
dyadic solenoid, $S_2\times
S_2$, which we represent as the inverse limit of the substitution
$$ \neswboxxx 
\to \begin{matrix}\nwseboxxx  \neswboxxx \cr
\neswboxxx  \nwseboxxx \end{matrix}, \qquad 
\nwseboxxx \to \begin{matrix}\nwseboxxx  \neswboxxx \cr
\neswboxxx  \nwseboxxx \end{matrix}.
$$
The approximant associated with this substitution is the torus $T^2=\R^2/L$,
where $L$ is the lattice spanned by $(1,1)$ and $(1,-1)$. In other words, 
$T^2$ is an infinite checkerboard modulo translational symmetry. Substitution
acts by doubling in each
direction, and the 2-dimensional dyadic solenoid is the inverse limit of
this torus under substitution.

The most intricate model, which we label with subscripts $(X,+)$,
comes from the substitution
$$ \nwbox{w}{x}{y}{z} \to \begin{matrix} 
\nwbox{w}{1}{y}{1}  \swbox{0}{x}{y}{0}\cr
\nebox{w}{0}{0}{z}  \nwbox{1}{x}{1}{z} \end{matrix},\qquad
\nebox{w}{x}{y}{z} \to \begin{matrix} 
\sebox{w}{0}{y}{0}  \nebox{1}{x}{y}{1}\cr
\nebox{w}{1}{1}{z}  \nwbox{0}{x}{0}{z} \end{matrix}, 
$$ 
\begin{equation}\label{subX+}\swbox{w}{x}{y}{z} \to \begin{matrix} 
\sebox{w}{0}{y}{0}  \swbox{1}{x}{y}{1}\cr
\swbox{w}{1}{1}{z}  \nwbox{0}{x}{0}{z} \end{matrix}, \qquad
\sebox{w}{x}{y}{z} \to \begin{matrix} 
\sebox{w}{1}{y}{1}  \swbox{0}{x}{y}{0}\cr
\nebox{w}{0}{0}{z}  \sebox{1}{x}{1}{z} \end{matrix},
\end{equation}
where each label $w,x,y,z$ can be either 0 or 1, and the two labels
adjacent to the head of an arrow are required to be the same. 

The remaining models are derived from the rules (\ref{subX+})
by deleting some information, either about edge labels or about
which way the arrows are pointing. 
The first letter ($X$, $/$, or
$0$) indicates whether we keep track of all arrows, just those in the
northeast or southwest direction, or none of the arrows. The second 
letter ($+$, $-$, or $0$) indicates whether we label all the edges,
just the horizontal edges, or no edges at all.

Specifically,
\begin{enumerate}
\item The $(X,-)$ substitution is the same as $(X,+)$, only without
any labels on the vertical edges. This eliminates the requirement that
the two labels at the head of an arrow agree. 
\item The $(X,0)$ substitution is the same as $(X,+)$ or $(X,-)$, only
with no edge labels at all. This is a version of 
the well-known chair substitution.
\item The $(/,+)$ substitution is the same as $(X,+)$, only with the
arrows pointing northwest and southeast identified. 
Specifically, the substitution is now
$$ \nebox{w}{x}{y}{z} \to \begin{matrix} 
\nwsebox{w}{0}{y}{0} \nebox{1}{x}{y}{1}\cr
\nebox{w}{1}{1}{z} \nwsebox{0}{x}{0}{z} \end{matrix}, \qquad
\swbox{w}{x}{y}{z} \to \begin{matrix} 
\nwsebox{w}{0}{y}{0} \swbox{1}{x}{y}{1}\cr
\swbox{w}{1}{1}{z} \nwsebox{0}{x}{0}{z} \end{matrix},
$$ 
\begin{equation*}\label{sub/+}
\nwsebox{w}{x}{y}{z} \to \begin{matrix} 
\nwsebox{w}{1}{y}{1}  \swbox{0}{x}{y}{0}\cr
\nebox{w}{0}{0}{z}  \nwsebox{1}{x}{1}{z} \end{matrix}.
\end{equation*}
On an double-headed arrow, either $w=y$ or $x=z$, while on a single-headed
arrow the labels at the head of the arrow must agree.
\item The $(/,-)$ substitution is the same as $(/,+)$, only with no labels
on the vertical edges.
\item The $(/,0)$ substitution is the same as $(/,+)$, only with no labels
on any edges. 
\item The $(0,+)$ substitution is 
\begin{equation*}\label{sub0+}
\neswbox{w}{x}{y}{z} \to \begin{matrix} 
\nwsebox{w}{0}{y}{0} \neswbox{1}{x}{y}{1}\cr
\neswbox{w}{1}{1}{z} \nwsebox{0}{x}{0}{z} \end{matrix},
\qquad 
\nwsebox{w}{x}{y}{z} \to \begin{matrix} 
\nwsebox{w}{1}{y}{1}  \neswbox{0}{x}{y}{0}\cr
\neswbox{w}{0}{0}{z}  \nwsebox{1}{x}{1}{z} \end{matrix}.
\end{equation*}
On each tile, either the labels at one head of the arrow
must agree, or the labels on the other head must agree. 
\item The $(0,-)$ substitution is the same as $(0,+)$, only without
any labels on the vertical edges. 
\end{enumerate}

\begin{remark}
The $(X,+)$ model is closely related to Goodman-Strauss' 
Trilobite and Crab (T\&C) tilings \cite{Chaim.crab}. 
The T\&C tilings can be written using the tiles 
of the $(X,+)$ model, only with local matching rules instead of 
a global substitution. The matching rules are:
\begin{enumerate}
\item Tiles meet full-edge to full-edge. 
\item Every edge has a 1 on one side and a 0 on the other. 
\item At vertices where three arrows comes in and the fourth goes out,
the labels near the head of the central incoming arrow are 1's, the labels
near the heads of the other incoming arrows are 0's, and the labels
near the tail of the outgoing arrow are 0's, and 
\item At all other vertices, the bottom edge of the northeast tile 
has the same marking as the bottom edge of the northwest tile, and the
left edge of the northeast tile has the same marking as the left edge
of the southeast tile. 
\end{enumerate}
All of these rules are satisfied by $(X,+)$ tilings, so the $(X,+)$ tiling
space is a subspace of the T\&C tiling space. 
Adapting an argument of Goodman-Strauss', one can show 
that all T\&C tilings are obtained from $(X,+)$ tilings by applying some
shears, either all along the NE-SW axis or all along the NW-SE axis. 
\end{remark}

\subsection{How the models are related}

The relations between the corresponding tiling spaces are summarized
in the diagram
\begin{equation}\label{relations}
\begin{CD} \Omega_{X,+} @>A>> \Omega_{/,+} @>A>> \Omega_{0,+} \\
@VVBV @VVBV @VVBV \\
\Omega_{X,-} @>A>> \Omega_{/,-} @>A>> \Omega_{0,-} \\
@VVAV @VVAV @ VVCV \\
\Omega_{X,0} @>A>> \Omega_{/,0} @>C>> \Omega_{0,0}, \\
\end{CD}
\end{equation}
where each map involves the erasing of some information about arrow or
edge markings. Each of these maps
is 1:1 outside of 
the orbit of a 1-dimensional tiling subspace. 
We can then apply Proposition \ref{Cohomology.of.suspension.applied}
to compute all of the quotient
tiling cohomologies for adjacent models.

In $\Omega_{X,+}$ there are 8 tilings that are fixed by the
substitution, corresponding to a single point in $S_2 \times
S_2$. The central patches of these tilings are:
$$A = \begin{matrix} \sebox{1}{0}{1}{0} \nebox{1}{1}{1}{1} \cr
\nebox{1}{1}{1}{1} \nwbox{0}{1}{0}{1} \end{matrix}, \qquad
B = \begin{matrix} \nwbox{1}{1}{1}{1} \swbox{0}{1}{1}{0} \cr
\nebox{1}{0}{0}{1} \nwbox{1}{1}{1}{1} \end{matrix},$$
$$C = \begin{matrix} \sebox{1}{0}{1}{0} \swbox{1}{1}{1}{1} \cr
\swbox{1}{1}{1}{1} \nwbox{0}{1}{0}{1} \end{matrix}, \qquad
D = \begin{matrix} \sebox{1}{1}{1}{1} \swbox{0}{1}{1}{0} \cr
\nebox{1}{0}{0}{1} \sebox{1}{1}{1}{1} \end{matrix},$$
$$E = \begin{matrix} \nwbox{1}{0}{1}{1} \nebox{1}{1}{1}{1} \cr
\swbox{1}{0}{0}{1} \sebox{1}{1}{0}{1} \end{matrix}, \qquad
F = \begin{matrix} \nwbox{1}{1}{1}{1} \nebox{0}{1}{1}{1} \cr
\swbox{1}{1}{0}{1} \sebox{0}{1}{0}{1} \end{matrix}, $$
$$G = \begin{matrix} \nwbox{1}{1}{1}{0} \nebox{0}{1}{1}{0} \cr
\swbox{1}{1}{1}{1} \sebox{0}{1}{1}{1} \end{matrix}, \qquad
H = \begin{matrix} \nwbox{1}{0}{1}{0} \nebox{1}{1}{1}{0} \cr
\swbox{1}{0}{1}{1} \sebox{1}{1}{1}{1} \end{matrix}. $$

These tilings are asymptotic in all directions except along the
coordinate axes and along the lines of slope $\pm 1$. In each of these
directions there are two possibilities, either corresponding to edge
labels along the axes or to the direction of the arrows along the main
diagonals. The map from  $\Omega_{X,+}$ 
to $S_2 \times S_2$ is thus
8:1 on the orbits of these tilings, 2:1 on tilings obtained by
translating these tilings in one of the eight principal directions and
taking limits, and 1:1 everywhere else.

The self-similar tilings with central patch $E$ and $F$ (henceforth called
the $E$ and $F$ tilings) differ only in the labels that appear on the
$y$ axis. In the tiling space $\Omega_{X,-}$, they are therefore identified,
as are their translational orbits. Likewise,
the $G$ and $H$ tilings are identified. The identifications for all 
the spaces are summarized in the table below.

\bigskip
\vbox{\center \begin{tabular}{||l|l||}
\hline
Model & Identifications \\
\hline\hline
$(X,+)$ & none \\
\hline
$(X,-)$ & $E=F$, $G=H$ \\
\hline
$(/,+)$ & $B=D$ \\
\hline
$(X,0)$ (chair) & $E=F=G=H$ \\
\hline
$(/,-)$& $B=D$, $E=F$, $G=H$ \\
\hline
$(0,+)$ & $A=C$, $B=D$ \\
\hline
$(/,0)$ & $B=D$, $E=F=G=H$ \\
\hline
$(0,-)$ & $A=C$, $B=D$, $E=F$, $G=H$ \\
\hline
$(0,0)$ (solenoid) & $A=B=C=D=E=F=G=H$\\
\hline
\end{tabular}}
\bigskip

Note that the closure of the set $\{A-\lambda(1,1)\}$, where $\lambda$
ranges over the real numbers, is a 1-dimensional tiling subspace of
$\Omega_{X,+}$ and is isomorphic to $S_2$. The closure of 
$\{C-\lambda(1,1)\}$ is a different copy of $S_2$. The closures of
$\{B - \lambda(1,-1)\}$, $\{D - \lambda(1,-1)\}$, $\{E-\lambda(1,0)\}$,
$\{E-\lambda(0,1)\}$, $\{F-\lambda(1,0)\}$, $\{F-\lambda(0,1)\}$,
$\{G-\lambda(1,0)\}$, $\{G-\lambda(0,1)\}$, $\{H-\lambda(1,0)\}$ and
$\{H - \lambda(0,1)\}$ are additional disjoint copies of $S_2$.  
Translating 
tilings $A$--$H$ in other directions is more complicated. For instance, the 
closure of $\{B - \lambda(1,1)\}$ consists of two copies of $S_2$ and 
a copy of $\R$ that connects them. One copy of $S_2$ comes from limits as
$\lambda \to +\infty$ and equals the closure of $\{C-\lambda(1,1)\}$, 
another comes from limits as $\lambda \to -\infty$ and equals the closure
of $\{A-\lambda(1,1)\}$,
and the interpolating line corresponds to finite values of $\lambda$. 

\begin{theorem}\label{adjacent}
The adjacent tiling spaces linked by maps in (\ref{relations})
have the following quotient cohomologies. When the factor map is
labeled ``A'', we have  
$H^1_Q= \Z$ and $H^2_Q=\Z[{\textstyle \half}]$, when it is labeled ``B'' we have 
$H^1_Q= \Z$ and $H^2_Q=\Z[{\textstyle \half}]\oplus \Z$, and when it is labeled
``C'' we have 
$H^1_Q= 0$ and $H^2_Q=\Z[{\textstyle \half}]\oplus \Z$. 
All adjacent pairs of spaces have $H^k_Q=0$ for $k \ne 1,2$.
\end{theorem}

\begin{proof}
We will show that all maps are covered by Proposition 
\ref{Cohomology.of.suspension.applied}, with $k=1$ and with
the pair $(X',Y')$ being either Degeneration A, B, or C, depending on 
the label of the arrow. Since in this case 
$H^m_Q(X,Y)=H^{m-1}_Q(X', Y')$, the theorem follows. 

Consider the map from $\Omega_{X,+}$ to $\Omega_{/,+}$. This map is
1:1 everywhere except that the closure of $\{B-\lambda(1,-1)\}$ is
identified with the closure of $\{D-\lambda(1,-1)\}$, and that finite
translates of these copies of $S_2$ are also identified. This is exactly
the situation of Proposition \ref{Cohomology.of.suspension.applied}, with
$\Gamma$ being the span of $(1,-1)$, with $X'\subset \Omega_{X,+}$ 
being the union of the closures
of $\{B-\lambda(1,-1)\}$ and $\{D-\lambda(1,-1)\}$, and with 
$Y'$ being their image after identification, and with the map between
them being Degeneration A. 
The remaining maps labeled ``A'' are similar. In each case we have 
two 1-dimensional tiling subspaces, each isomorphic to $S_2$, that are
identified.

Next consider the map from $\Omega_{X,+}$ to $\Omega_{X,-}$. This is
1:1 except that $E-v$ and $F-v$ are identified for all $v\in\R^2$,
$G-v$ and $H-v$ are identified for all $v\in\R^2$, as are all pairs of
tilings obtained as limits of translations of these pairs. 
Note that
$E$ and $H$ are asymptotic under translation in both vertical
directions, so the closure of the union of $\{E-\lambda(0,1)\}$
and $\{H - \lambda(0,1)\}$ is not two solenoids. Rather, it is a 
copy of $\Omega_{PD}$. The closure of the union of $\{F-\lambda(0,1)\}$
and $\{G-\lambda(0,1)\}$ is another copy of $\Omega_{PD}$, so
$X'=\Omega_{PD} \times \{1,2\}$. The image $Y'$ of $X'$ is a single
copy of $\Omega_{PD}$ in $\Omega_{X,-}$, corresponding to the vertical
orbit closure of $\{E=F, G=H\}$. This is Degeneration B. 

The same analysis applies to the other ``B'' maps, 
from $\Omega_{/,+}$ to $\Omega_{/,-}$ 
and from $\Omega_{0,+}$ to $\Omega_{0,-}$.

The map from $\Omega_{/,0}$ to $\Omega_{0,0}$ involves 
the identification of $A$, $B$, $C$, and $E$, where we already have 
$B=D$ and $E=F=G=H$. As noted above, the closure of $\{B-\lambda(1,1)\}$
already contains the closures of $\{A-\lambda(1,1)\}$ and 
$\{C-\lambda(1,1)\}$. So does the closure of $\{E-\lambda(1,1)\}$. Let 
$X'$ be the union of these four closures. $X'$ consists of two solenoids 
and two connecting copies of $\R$, one running from the first solenoid
to the second, and the other running from the second solenoid to the first. 
This is the
inverse limit of $\Gamma_{TM'}$ under a map the doubles each circle and
preserves the connections between them. The image $Y'$ of $X'$ in 
$\Omega_{0,0}$ consists of a single copy of $S_2$, and the map
from $X'$ to $Y'$ is Degeneration C. 
The map from $\Omega_{0,-}$ to $\Omega_{0,0}$ is similar, only with
horizontal translations instead of diagonal, 
and with the identification of $A$, $B$, $E$, and $G$,
instead of $A$, $B$, $C$, and $E$. 
\end{proof}

\subsection{Torsion in quotient cohomology}
There is no torsion in the one-step quotient cohomology of Theorem
\ref{adjacent}. However, there is 3-torsion in 
$H^2_Q(\Omega_{X,0}, \Omega_{0,0})$. In this subsection we explore how
this comes about. The solenoid $\Omega_{0,0}$ has $H^1=\Z[{\textstyle \half}]^2$ and 
$H^2=\Z[{\textstyle \quarter}]$.\footnote{$\Z[{\textstyle \quarter}]$ is of course isomorphic to $\Z[{\textstyle \half}]$,
but we write 
${\textstyle \quarter}$ to emphasize that substitution is multiplication by 4,
and not by 2.}

In the chair space $\Omega_{X,0}$, 
tiles aggregate into 3-tile groups that look like
an L or a  chair \cite{Robbie}. 
The center of each chair is an arrow tile whose head is flanked by
two other arrowheads, as with the lower left tile of patch $A$, the lower 
right tile of patch $B$, the upper right tile of patch $C$ and the upper left
tile of patch $D$.  The heads of arrows of 
tiles that are not in the center of a chair are flanked by an arrowhead and
the tail of an arrow, rather than by two arrowheads. The position of a
tile within its chair can thus be determined by the local patterns of arrows. 

Consider a (pattern-equivariant) 
cochain $\alpha$ that evaluates to 1 on the middle tile of each
chair, but to zero on the outer two tiles of each chair. 
$3\alpha$ is cohomologous to a cochain $\beta$ 
that evaluates to 1 on every tile. The cochain $\beta$ is the pullback of 
the generator of $H^2(\Omega_{0,0})=\Z[{\textstyle \quarter}]$. Thus $[\alpha]$ is a 
non-trivial 3-torsion element in $H^2_Q(\Omega_{X,0}, \Omega_{0,0})$. 

Applying the long exact sequence (\ref{LEST}) 
to the triple $(\Omega_{X,0}, \Omega_{/,0}, \Omega_{0,0})$, 
we get
$$ 0 \to H_Q^1(\Omega_{X,0},\Omega_{0,0}) \to \Z \xrightarrow{\delta}
\Z[{\textstyle \half}] \oplus \Z \to H_Q^2(\Omega_{X,0}, \Omega_{0,0}) \to \Z[{\textstyle \half}] \to 0.$$
For torsion to appear in $H^2_Q(\Omega_{X,0}, \Omega_{0,0})$, 
the map $\delta$ must be injective. If fact, it is multiplication by $(0,3)$, 
and $H^2(\Omega_{X,0}, \Omega_{/,0}) = \Z[{\textstyle \half}]^2 \oplus\Z_3$. 

There is no torsion in the absolute cohomology of $\Omega_{/,0}$ or 
$\Omega_{X,0}$. We compute $H^k(\Omega_{/,0})$ from the long exact sequence of 
the pair $(\Omega_{/,0}, \Omega_{0,0})$. 
Since $H^1_Q(\Omega_{/,0}, \Omega_{0,0})=0$,
we have $H^1(\Omega_{/,0})=H^1(\Omega_{0,0})= \Z[{\textstyle \half}]^2$ and 
$$ 0 \to \Z[{\textstyle \quarter}] \to H^2(\Omega_{/,0}) \to \Z[{\textstyle \half}] \oplus \Z \to 0.$$
This sequence must split, since any preimage of a generator of $\Z[{\textstyle \half}]$ 
must be infinitely divisible by 2, so $H^2(\Omega_{/,0})= \Z[{\textstyle \quarter}]
\oplus \Z[{\textstyle \half}] \oplus \Z$.

In the long
exact sequence of the pair $(\Omega_{X,0}, \Omega_{/,0})$, 
$$0 \!\to\! \Z[{\textstyle \half}]^2 \!\to\! H^1(\Omega_{X,0}) \!\to\! \Z 
\!\xrightarrow{\delta}\!
\Z[{\textstyle \quarter}] \oplus \Z[{\textstyle \half}] \oplus \Z \!\to\! H^2(\Omega_{X,0}) \!\to\! \Z[{\textstyle \half}] 
\!\to\! 0,$$
the coboundary
map $\delta$ is multiplication
by $(-1,0,3)$. The element $(0,0,1)$, which can be represented by 
the cochain $\alpha$, is no longer a 
torsion element in
the cokernel. Rather, 3 times this element is equivalent to $(1,0,0)$, 
a generator of the 
original $\Z[{\textstyle \quarter}]$. We denote this 3-fold extension of $\Z[{\textstyle \quarter}]$ as 
$\frac{1}{3}\Z[{\textstyle \quarter}]$.

Since $\delta$ is an injection, $H^1(\Omega_{X,0})=\Z[{\textstyle \half}]^2$, with
generators that are pullbacks of the generators of $H^1(\Omega_{0,0})$,  
while $H^2(\Omega_{X,0})= \frac{1}{3}\Z[{\textstyle \quarter}] \oplus \Z[{\textstyle \half}]^2$. These results
for the chair cohomology are not new, but the derivation via quotient 
cohomology helps to elucidate each term.

\subsection{Absolute cohomologies}

We continue the process of finding the absolute cohomologies of the nine
models, and then the quotient cohomology of each model relative to the 
solenoid $\Omega_{0,0}$, by repeatedly combining the one-step quotient
cohomologies of Theorem \ref{adjacent}.

For each adjacent pair $(X,Y)$, it is possible to represent a
generator of $\Z[{\textstyle \half}] \subset H^2_Q(X,Y)$ by a cochain on $X$, which
then generates a $\Z[{\textstyle \half}]$ subgroup of $H^2(X)$.  These
representatives are described as follows: When $X$ is a / model and
$Y$ is a 0 model, the representative evaluates to +1 on every tile
whose arrow points northeast, -1 on every tile whose arrow points
southwest, and 0 on 2-headed arrows. When $X$ is an $X$ model and $Y$
is a / model, the representative evaluates to +1 on tiles whose arrows
point southeast and -1 on tiles whose arrows point northwest. This
representative, combined with the previous one, simply counts the
vector sum of all the arrows.  When $X$ is a $-$ model and $Y$ is a 0
model, the representative counts the label on the top edge of each
tile minus the label on the bottom edge. Likewise, when $X$ is a $+$
model and $Y$ is a $-$ model, the representative counts the label on
the right edge minus the label on the left. The reader can check that
whenever there are doubly-asymptotic tilings in $X$ that are
identified in $Y$, the representative evaluates differently on the
tiles in the central strip where the two tilings are different.  All
four of these representatives double with substitution, and so
generate copies of $\Z[{\textstyle \half}]$. \footnote{The attentive reader may ask
  whether our representatives could correspond to multiples of the
  generators of $\Z[{\textstyle \half}] \subset H^2_Q(X,Y)$, rather than to the
  generators themselves. Eliminating this possibility requires working
  carefully through the details of degenerations A, B and C, together
  with the proof of Proposition
  \ref{Cohomology.of.suspension.applied}.}

Since a generator of $\Z[{\textstyle \half}] \subset H^2_Q(X,Y)$ can be represented by
an element of $H^2(X)$ that is infinitely divisible by 2, the exact sequence
$$ 0 \to coker(\delta) \to H^2(X)\to H^2_Q(X,Y) \to 0 $$
splits, where 
$\delta: H^1_Q(X,Y) \to H^2(Y)$ is the coboundary map in the
long exact sequence (\ref{LES1}). For the maps marked A
and B, $H^1_Q(X,Y)=\Z$. We must determine whether this $\Z$ contributes to
$H^1(X)$ (if $\delta$ is the zero map) or cancels part of $H^2(Y)$.
Since $\delta$ commutes with substitution, an element of a $\Z$ term 
can never map to a nonzero
element of $\Z[{\textstyle \half}]$ or $\Z[{\textstyle \quarter}]$, or to a combination of the two --- 
cancellations are only possible when
$\Z$ terms of $H^2(Y)$ are involved. 

In going from  $\Omega_{X,0}$ to $\Omega_{X,-}$, and then from 
$\Omega_{X,-}$ to $\Omega_{X,+}$, there is nothing to cancel, as there are no
$\Z$ terms in $H^2(Y)$.
This implies that 
$$ H^2(\Omega_{X,+}) = \frac{1}{3}\Z[{\textstyle \quarter}] \oplus \Z[{\textstyle \half}]^4 \oplus \Z, \qquad
H^1(\Omega_{X,+}) = \Z[{\textstyle \half}]^2 \oplus \Z^2. $$

Note that all paths from $\Omega_{X,+}$ to $\Omega_{0,0}$ involve two
A degenerations, one B degeneration and one C degeneration. Since
one such path (namely $\Omega_{X,+} \to \Omega_{X,-} \to \Omega_{X,0}
\to \Omega_{/,0} \to \Omega_{0,0}$) involves a cancellation at one step, all
such paths must involve exactly one cancellation. 

These cancellations occur in the maps from $\Omega_{X,-}$ to $\Omega_{/,-}$
and from $\Omega_{X,+}$ to $\Omega_{/,+}$, and 
are identical in form to the cancellation that occurs in going from 
$\Omega_{X,0} \to \Omega_{/,0}$. In each case, the generators of 
$H^1_Q(X,Y)$ are cochains
that only see the structure of the arrows, not the edge markings, and
one can check that the coboundary map is nonzero.

Another way to see that cancellations occur in these maps, and only in 
these maps, is to work out the cohomology of $\Omega_{X,+}$ in detail, 
either via $H^*(\Omega_{X,-})$ or directly. 
Every element of $H^1(\Omega_{X,+})$ can be represented by a cochain that
is the pullback of a 
cochain on $\Omega_{/,+}$, implying that $H^1(\Omega_{/,+})$ surjects on
$H^1(\Omega_{X,+})$.  Thus the map from
$H^1(\Omega_{X,+})$ to $H^1_Q(\Omega_{X,+}, \Omega_{/,+})=\Z$ is the zero map, 
so $\delta$ is injective and there is a cancellation in going from 
$\Omega_{X,+}$ to $\Omega_{/,+}$. There then cannot be any cancellations along
any path from $\Omega_{/,+}$ to $\Omega_{0,0}$, and there must be 
a cancellation in going from $\Omega_{X,-}$ to $\Omega_{/,-}$. 

This determines all of the
remaining cohomologies, both absolute and relative to $\Omega_{0,0}$.
We summarize these calculations in two theorems:

\begin{theorem}\label{absolute.cohomologies}
The absolute cohomologies of the nine models are given as follows. All
models have $H^0=\Z$. The first cohomology is given by 
\begin{equation}\label{H1absolute}
\begin{CD} 
\Z[{\textstyle \half}]^2 \oplus \Z^2 @({A^*}(( 
\Z[{\textstyle \half}]^2 \oplus \Z^2 @({A^*}(( 
\Z[{\textstyle \half}]^2 \oplus \Z  \\
@AA{B^*}A @AA{B^*}A @AA{B^*}A \\
\Z[{\textstyle \half}]^2 \oplus \Z @({A^*}(( 
\Z[{\textstyle \half}]^2 \oplus \Z @({A^*}(( 
\Z[{\textstyle \half}]^2   \\
@AA{A^*}A @AA{A^*}A @AA{C^*}A \\
\Z[{\textstyle \half}]^2  @({A^*}(( 
\Z[{\textstyle \half}]^2 @({C^*}(( 
\Z[{\textstyle \half}]^2,   \\
\end{CD}
\end{equation}
where the positions correspond to the positions in (\ref{relations}). 
The second cohomology is given by 
\begin{equation}\label{H2absolute}
\begin{CD} 
\frac{1}{3}\Z[{\textstyle \quarter}]\oplus \Z[{\textstyle \half}]^4 \!\oplus\! \Z @({A^*}(( 
\Z[{\textstyle \quarter}]\oplus \Z[{\textstyle \half}]^3 \oplus \Z^2 @({A^*}(( 
\Z[{\textstyle \quarter}]\oplus \Z[{\textstyle \half}]^2 \oplus \Z^2\\
@AA{B^*}A @AA{B^*}A @AA{B^*}A \\
\frac{1}{3}\Z[{\textstyle \quarter}]\oplus \Z[{\textstyle \half}]^3 @({A^*}(( 
\Z[{\textstyle \quarter}]\oplus \Z[{\textstyle \half}]^2 \oplus \Z @({A^*}(( 
\Z[{\textstyle \quarter}]\oplus \Z[{\textstyle \half}] \oplus \Z\\
@AA{A^*}A @AA{A^*}A @AA{C^*}A \\
\frac{1}{3}\Z[{\textstyle \quarter}]\oplus \Z[{\textstyle \half}]^2 @({A^*}(( 
\Z[{\textstyle \quarter}]\oplus \Z[{\textstyle \half}] \oplus \Z @({C^*}(( 
\Z[{\textstyle \quarter}]. \\
\end{CD}
\end{equation}
\end{theorem}

\begin{theorem}\label{quotients}
The quotient cohomologies of the nine models, relative to the
solenoid $\Omega_{0,0}$, are given as follows. 
The first cohomology is given by 
\begin{equation}\label{H1quotient}
\begin{CD} 
\Z^2 @({A^*}(( 
\Z^2 @({A^*}(( 
\Z  \\
@AA{B^*}A @AA{B^*}A @AA{B^*}A \\
\Z @({A^*}(( 
\Z @({A^*}(( 
0   \\
@AA{A^*}A @AA{A^*}A @AA{C^*}A \\
0 @({A^*}(( 
0 @({C^*}(( 
0.   \\
\end{CD}
\end{equation}
The second cohomology is given by 
\begin{equation}\label{H2quotient}
\begin{CD} 
\Z_3 \oplus \Z[{\textstyle \half}]^4 \oplus \Z @({A^*}(( 
\Z[{\textstyle \half}]^3 \oplus \Z^2 @({A^*}(( 
\Z[{\textstyle \half}]^2 \oplus \Z^2\\
@AA{B^*}A @AA{B^*}A @AA{B^*}A \\
\Z_3 \oplus \Z[{\textstyle \half}]^3 @({A^*}(( 
\Z[{\textstyle \half}]^2 \oplus \Z @({A^*}(( 
\Z[{\textstyle \half}] \oplus \Z\\
@AA{A^*}A @AA{A^*}A @AA{C^*}A \\
\Z_3 \oplus \Z[{\textstyle \half}]^2 @({A^*}(( 
\Z[{\textstyle \half}] \oplus \Z @({C^*}(( 
0. \\
\end{CD}
\end{equation}
\end{theorem}


\section{Tilings of finite type}\label{finite}

In 1989, Mozes \cite{Mozes} proved a remarkable theorem  relating
substitution subshifts in 2 or more dimensions to subshifts of finite type. 
Radin \cite{pinwheel} applied Mozes' ideas to the pinwheel tiling 
and Goodman-Strauss \cite{Chaim.annals} 
generalized them to tilings in general.  Although not phrased in this
language, Goodman-Strauss' results imply the following theorem:

\begin{theorem}\label{Finite.type} 
Let $\sigma$ be a tiling substitution in 2 dimensions (or more),
and let $\Omega_\sigma$ be the corresponding tiling space. Suppose that the
tiles are polygons that meet full-edge to full edge.\footnote{Or in higher 
dimensions, polyhedra that meet full-face to full face. These assumptions
can actually be relaxed considerably.} Then there exists a tiling space
$\Omega_{FT}$ whose tilings are defined by local matching rules, and a factor
map $f: \Omega_{FT} \to \Omega_\sigma$ such that (1) $f$ is everywhere 
finite:1, and 1:1 except on a set of measure zero, and (2) the set where
$f$ is not injective maps to tilings in $\Omega_\sigma$ containing 
two or more infinite-order supertiles. 
\end{theorem}

For measure-theoretic purposes, $\Omega_{FT}$ and $\Omega_{\sigma}$ are the
same, so the extensive analysis of substitution tilings can give us 
measure-theoretic information about some finite-type tiling spaces. 
For topological purposes, however, $\Omega_{FT}$ and $\Omega_\sigma$ are 
different, and it is known \cite{RS} that
some substitution tiling spaces are not homeomorphic to {\em any} tiling 
spaces of finite type. 

If the factor map $f$ failed to be 1:1 only over tilings in 
$\Omega_\sigma$ where infinite-order supertiles met along
horizontal boundaries, then we could apply Proposition
\ref{Cohomology.of.suspension.applied} to the pair $(\Omega_{FT},
\Omega_{\sigma})$. $Y'$ would be the space of tilings where that meeting
is exactly on the horizontal axis, and $X'$ would be the pre-image of those
tilings in $\Omega_{FT}$. 

Of course, substitution tilings have supertiles meeting along boundaries 
pointing in several directions. Still, as long as there are only finitely
many such directions (this excludes examples like the pinwheel tiling), 
we can take the quotient of $\Omega_{FT}$ one direction at a time. This is 
essentially what we did with the nine chair-like models, where the factors 
from $+$ to $-$, from $-$ to $0$, from  $X$ to $/$, and from $/$ to $0$ 
involve dismissing information along infinite vertical, horizontal, and
diagonal lines. There will be many possible orders in which we take
quotients, and we will have to choose a path from $\Omega_{FT}$ to
$\Omega_\sigma$ that makes the calculation as simple as possible.

There are complications involving tilings where more than two
infinite-order supertiles meet at a vertex. Sometimes we will have to
dismiss information specific to a finite collection of orbits, an
application of Proposition \ref{Cohomology.of.suspension.applied} with
$k=0$ rather than $k=1$.  Perhaps the spaces intermediate between
$\Omega_{FT}$ and $\Omega_{\sigma}$ will not have a ready description
as tiling spaces, but only as quotients of tiling spaces or as
extensions of tiling spaces.

These complications should not deter us. As long as there is a path
from $\Omega_{FT}$ to $\Omega_{\sigma}$, it should be possible to
compute one-step quotient cohomologies. These can then be combined,
either through repeated application of long exact sequences of pairs
or triples, or via a spectral sequence \cite{Mccleary}.

Extremely little is currently known about the topology of tiling
spaces of finite type.  Our hope, and belief, is that quotient
cohomology will open up finite type tiling spaces for topological
exploration.

\subsection*{\bf Acknowledgments.}We thank Andrew Blumberg, 
John Hunton, John McCleary, Tim Perutz and Bob Williams for helpful
discussions, and thank Margaret Combs for TeXnical and artistic help.
We also thank C.I.R.M., where part of this work was done. The
work of L.S. is partially supported by the National Science Foundation
under grant DMS-0701055.

\bigskip

\end{document}